\documentclass[12pt]{article}
\usepackage{amsfonts,amssymb}
\usepackage{latexsym}
\usepackage[usenames,dvipsnames]{color}
\usepackage{subfigure}
\usepackage{amsmath}
\usepackage{amsthm}
\usepackage{fullpage}
\usepackage{multirow}

\newtheorem{thm}{Theorem}[section]
\newtheorem{exa}[thm]{Example}
\newtheorem{lem}[thm]{Lemma}

\newcommand{\Z}{\mathbb{Z}}
\newcommand{\N}{\mathbb{N}}

\begin{document}

\title{Roman and Vatican Crossover Designs  }

\author{M.~A.~Ollis\footnote{Email address: \texttt{matt@marlboro.edu} }   \\
    {\it Marlboro College, P.O.~Box A, Marlboro,} \\    
    {\it Vermont 05344, USA.} }

\date{}

\maketitle

\begin{abstract}
Latin squares with a balance property among adjacent pairs of symbols---being ``Roman" or ``row-complete"---have long been used as uniform crossover designs with the number of treatments, periods and subjects all equal.   This has been generalized in two ways: to crossover designs with more subjects and to balance properties at greater distances.  We consider both of these simultaneously, introducing and constructing {\em Vatican designs}: these  have  $\ell t$ subjects, $t$ periods and treatments, and, for each~$d$ in the range~$1 \leq d < t$, the number of times that any subject receives treatment~$j$ exactly~$d$ periods after receiving treatment~$i$ is at most~$\ell$.  Results include showing the existence of Vatican designs when $t+1$ is prime (for any~$\ell$), when $5 \leq t \leq 14$ and $\ell >1$, and when $t \in \{ 3,15 \}$ and $\ell$ is even.

\vspace{2mm}
\end{abstract}

\section{Introduction}\label{sec:intro}

In the theory of experimental designs, a {\em crossover design} is one in which the experimental subjects each receive a test treatment in each of multiple periods.   Suppose there are~$n$ subjects, $t$ treatments and $p$ periods.  We shall display such a design as a $n \times p$ array~$D$ in which $D_{ij}$ represents the treatment received by subject~$i$ in period~$j$.  

We shall limit ourselves to {\em uniform} crossover designs: those in which each treatment occurs the same number of times in each row and the same number of times in each column.  We further limit our investigation to those in which~$p=t$, so each subject receives each treatment exactly once.

In a uniform crossover design, for an ordered pair~$(x,y)$ of treatments define $o(x,y)$ to be the number of times $y$ occurs immediately after~$x$.  If~$o(x,y)$ is constant across all ordered pairs of distinct elements, then the design is {\em balanced}.  Balance is desirable in situations where one treatment might have a ``carry-over" eftect to the next time period.  A survey of the theory of such designs is~\cite{BJ08}.

A {\em Latin square} is a crossover design with $n=p=t$.  If it is balanced, then it is {\em Roman} or {\em row-complete}.  We extend the domain of this definition and call any balanced crossover design {\em Roman}.  

In the study of Roman squares, a stronger notion of balance was introduced by Etzion, Golomb and Taylor~\cite{EGT89}. We also extend this to designs.  Let $o_i(x,y)$ be the number of times that treatment~$y$ occurs exactly~$i$ time periods after~$x$ (so $o_1(x,y) = o(x,y)$ as defined above).   A uniform design with~$p=t$ is a Roman-$k$ design if $o_i(x,y) \leq n/t$ for all~$i \leq k$.  For Latin squares, this says that each ordered pair of distinct treatments occurs at distance~$i$ at most once in the square for each~$i \leq k$.   Again mirroring the definitions for Latin squares, call a uniform design with~$p=t$ {\em Vatican} if it is Roman-$(t-1)$ (that is, the balance property holds at all possible distances).

Figure~\ref{fig:56} shows a Roman (but not Roman-2) design with $n=t=6$ and a Vatican design with $n=2t=10$.  Clearly, there is some regularity to their construction; we explore this in the next two sections.

\begin{figure}
\caption{A Roman design and a Vatican design}\label{fig:56}
\begin{center}
$
\begin{array}{cccccc}
0&5&1&4&2&3 \\
1&0&2&5&3&4 \\
2&1&3&0&4&5 \\
3&2&4&1&5&0 \\
4&3&5&2&0&1 \\
5&4&0&3&1&2
\end{array}
$
\hspace{20mm}
$
\begin{array}{ccccc}
0&1&3&4&2 \\
1&2&4&0&3 \\
2&3&0&1&4 \\
3&4&1&2&0 \\
4&0&2&3&1 \\
0&4&2&1&3 \\
1&0&3&2&4 \\
2&1&4&3&0 \\
3&2&0&4&1 \\
4&3&1&0&2 
\end{array}
$
\end{center}
\end{figure}

In the next section we show how to build designs from sequences of group elements and in Section~\ref{sec:tuples} we employ and expand the theory of ``terraces" to build these sequences.   Ultimately, and in conjunction with existing results on Vatican squares, we are able to prove:

\begin{thm}\label{th:vat}
There is an $\ell t \times t$ Vatican design in each of the following cases:
\begin{itemize}
\item $t+1$ is prime,
\item $5 \leq t \leq 14$ and $\ell >1$,
\item $t \in \{ 3, 15 \} $ and $\ell$ is even,
\item $t$ is prime and $\ell$ is a multiple of $t-1$.
\end{itemize}
\end{thm}

We also give various stronger results than the fourth item of Theorem~\ref{th:vat} for some prime numbers~$t$ with~$t \leq 281$.

\section{From groups to designs}\label{sec:bricks}

The general method of construction is to form the desired $n \times t$ design, where~$n = \ell t$, by taking~$\ell$ Latin squares of order~$t$.  Each of the Latin squares is the Cayley table of a group.  Most of the results can be achieved with cyclic groups, which we write as  $\Z_t = \{ 0, 1, \ldots, t-1 \}$ with the operation of addition modulo~$t$,  but we need the more general theory for some orders.

The following result means that we can limit our attention to small values of~$\ell$.

\begin{lem}\label{lem:stack}
If there is an~$\ell t \times t$ Roman-$k$ design for each $\ell \in \{\ell_1, \ldots, \ell_m \}$, then there is an $n \times t$ Roman-$k$ design for $n  = c_1\ell_1 + \cdots c_m\ell_m$ for any choice of non-negative integers~$c_1, \ldots, c_m$.
\end{lem} 

\begin{proof}
Simply stack $c_i$ copies of the $\ell_i t \times t$ Roman-$k$ design for each~$i$. 
\end{proof}

Thus Roman-$k$ and Vatican squares are the ideal building block.  Existing results for these objects give many orders of Roman-$k$ and Vatican designs:

\begin{thm}\label{th:squarecase}
There is an~$\ell t \times t$ Roman-$k$ design for all~$\ell \in \N$ in each of the following cases:
\begin{itemize}
\item $k = t-1$ (i.e. the design is Vatican) and $t+1$ is prime,
\item $k=2$ and $t = 2q$ for some prime~$q$ with $q \equiv 7 \pmod{12}$ or $q \equiv 5 \pmod{24}$,
\item $k=2$ and $t$ is even with $t \leq 50$, 
\item $k=2$ and $t=21$,
\item $k=1$ and $t$ is composite.
\end{itemize}

\end{thm}

\begin{proof}
In each case there is a Roman-$k$ square of order~$t$~\cite{Anderson91, CL94, GET90, Higham98, Ollis19, Williams49}.  Applying Lemma~\ref{lem:stack} gives the result. 
\end{proof}

Rather than trying to construct more Roman-$k$ or Vatican squares, which seems to be a difficult problem, we take a different approach. Observe that for~$\ell >1$ we may write~$\ell = 2c_1 + 3c_2$ for some $c_1, c_2 \geq 0$ and so to construct an $\ell t \times t$ Roman-$k$ design for all $\ell >1$ it suffices to construct them for $\ell =2$ and~$\ell =3$. This is the essence of how the following result is proved (we give the proof early in the next section when we have more machinery available).

\begin{thm}\label{th:roman1}{\rm \cite{Prescott99, Williams49}}
There is an $\ell t \times t$ Roman design for $\ell > 1$ and any value of~$t$.
\end{thm}

Let~$G$ be a group of order~$t$ and let ${\bf a} = (a_1, \ldots, a_t)$ be an ordering of the elements of~$G$.  Let $g{\bf a} =  (ga_1, \ldots, ga_t)$ and define $L({\bf a})$ be a Latin square with rows $\{ g{\bf a} : g \in G \}$ (the order of the rows does not concern us).  The squares we use to build design all have this form.

Given such a sequence~${\bf a}$, define its {\em quotient triangle} $(T_1, T_2, \ldots, T_{t-1})$ by:
$$
\begin{array}{rlllll}
T_1: & a_1^{-1}a_2, & a_2^{-1}a_3, &  a_3^{-1}a_4, & \ldots, & a_{t-1}^{-1}a_t \\
T_2: &                 & a_1^{-1}a_3, & a_2^{-1}a_4, & \ldots ,& a_{t-2}^{-1}a_t \\
T_3: &                 &                  & a_1^{-1}a_4, & \ldots, & a_{t-3}^{-1}a_t \\
  && && \vdots & \\
T_{t-1}:  &&&&&  a_1^{-1}a_t \\                
\end{array}
$$
When~$G$ is abelian, we usually use additive notation and call the quotient triangle the {\em difference triangle}.

These quotients control the neighbor properties we are interested in.  For each occurrence of $x$ in the $i$th line~$T_i$ of a quotient triangle, an ordered  pair $(g,h)$ with $g^{-1}h = x$ appears once at distance~$i$ among the rows of~$L({\bf a})$.  This motivates the following definition.

Let~$G$ be a group of order~$t$.
Let~${\bf A} = ({\bf a_1}, \ldots, {\bf a_\ell})$ where each~${\bf a_i} = (a_{i1}, a_{i2}, \ldots, a_{i,t})$ is an arrangement of the elements of $G$.   Let~$T_i$ be the quotient triangle for~${\bf a_i}$, with lines~$T_{i1}, T_{i2}, \ldots T_{i,t-1}$ and let~$U_i$ be the concatenation of the $i$th lines of the quotient triangles $T_{1,i}, \ldots, T_{\ell,i}$.  If each non-identity element of~$G$ appears at most~$\ell$ times in each~$U_i$ for $1 \leq i \leq k$, then $A$ is a {\em Roman-$k$ $\ell$-tuple}.  Call a Roman-$(t-1)$  $\ell$-tuple a {\em Vatican $\ell$-tuple}.  We refer to 1-, 2- and 3-tuples, with which we will mostly be working, as {\em singletons}, {\em pairs} and {\em triples} respectively.

A Roman-$k$ singleton is known in the literature as a {\em directed $T_k$-terrace}.

\begin{exa}\label{ex:vat6}
A Vatican singleton for $\Z_6$:
$$
\begin{array}{rllllll}
{\bf a_1}: & 0 & 4 & 5 & 2 & 1 & 3 \\
\hline 
T_1 : & & 4 & 1 & 3 & 5 & 2 \\
T_2: &&&  5 & 4 &  2 & 1 \\
T_3: &&&&  2 & 3 & 4 \\
T_4: &&&&& 1 & 5 \\
T_5: &&&&&& 3 \\
\end{array}
$$
\end{exa}

\begin{thm}\label{th:lt}
If a group of order~$t$ has a Roman-$k$ $\ell$-tuple then there is a $\ell t \times t$ Roman-$k$ design.
\end{thm}

\begin{proof}
Let $({\bf a_1}, \ldots, {\bf a_\ell})$ be a Roman-$k$ $\ell$-tuple for~$G$ and consider the design~$D$ obtained by stacking $L({\bf a_1}), \ldots L({\bf a_{\ell}})$.  We have an occurrence of the ordered pair $(g,h)$ of distinct elements of~$G$ at distance~$i$ in a row of~$D$ exactly once for every occurrence of $g^{-1}h$ in the $i$th line of the quotient triangle.  Hence there are at most~$\ell$ occurrences of the each such pair $(g,h)$ at distance~$i$.
\end{proof}

The challenge now is to construct these $\ell$-tuples.

\section{Constructing Roman-$k$ and Vatican $\ell$-tuples}\label{sec:tuples}

In order to construct $2t \times t$ Roman designs, Williams introduced an example of what came to be known as a ``terrace."  We generalize this approach to create single sequences from which all of the sequences of an $\ell$-tuple can in some cases be constructed.

Let~$G$ be a group of order~$t$ with an automorphism~$\alpha$ of order~$\ell$.  For $g \in G$ define the cycle of $g$ under $\alpha$ as $\bar g = \{ \alpha^r (g) : 1 \leq r \leq \ell  \}$.  Let~${\bf a} = (a_1, \ldots, a_t)$ be an ordering of the elements of~$G$ with quotient triangle $(T_1, \ldots, T_{t-1})$.  For each non-identity element~$g$, if the number of times an element of $\bar g$ occurs in $T_i$ is at most $| \bar g |$ for $1 \leq i \leq k$ then~${\bf a}$ is a {\em Roman-$k$ pseudoterrace} with respect to~$\alpha$.

If ${\bf a}$ is a Roman-$k$ pseudoterrace then ${\bf A} = ({\bf a_1}, \ldots, {\bf a_\ell})$, where ${\bf a_r} = (\alpha^r (a_1), \ldots, \alpha^r (a_t))$ for each~$r$, is a Roman-$k$ $\ell$-tuple as $U_i$, the concatentation of the $j$th lines of the difference triangles, consists of the elements of the form $\alpha^r(a_i^{-1}a_{i+j})$ for $1 \leq r \leq \ell$.

If $\alpha$ is given by $x \mapsto x^{-1}$ (in which case the group is abelian) then a Roman pseudoterrace is the same as a {\em terrace} as defined in~\cite{Bailey84}.  The sequence 
$$(0, t-1, 1, t-2, \ldots )$$
is a terrace for~$\Z_t$, a construction first given by Walecki for even~$t$ (in which case it is a directed terrace or, in the vocabulary of this paper, a Roman singleton) and Williams for odd~$t$~\cite{Alspach08,Williams49}.  (Historical note: Williams and others did not use this method to construct Roman pairs from terraces. Instead they use that when~${\bf a}$ is a terrace then ${\bf a}$ along with the reverse of~${\bf a}$ is a Roman pair.)  

We can now prove Theorem~\ref{th:roman1}:

\begin{proof}[Proof of Theorem~\ref{th:roman1}]
By Theorem~\ref{th:lt} and Lemma~\ref{lem:stack}, it is sufficient to find a Roman pair and triple for all values of~$t$.  

As we have just observed, the Walecki construction gives a Roman singleton (and hence also a pair and a triple) for even~$t$ and a Roman pair when $t$ is odd.  We therefore only require a Roman triple for odd~$t$.  Consider the following sequences, where semicolons are used to separate the patterns:
\begin{eqnarray*}
{\bf a_1} &=& \left( 0 ; 1, t-2, 3, t-4, \ldots, \frac{t+(-1)^{\lfloor t/2 \rfloor}}{2}  ;  \frac{t+(-1)^{\lfloor t/2 \rfloor}}{2} +1, \ldots, t-3, 4, t-1, 2  \right)  \\
{\bf a_2} &=&  \left( 0 ; 2, t-1, 4, t-3, \ldots,  \frac{t+(-1)^{\lfloor t/2 \rfloor}}{2} +1 ;   \frac{t+(-1)^{\lfloor t/2 \rfloor}}{2} , \ldots, t-4, 3 , t-2, 1 \right) \\
{\bf a_3} &=&  \left( 0,  t-1 ;  t-2, 2, t-4, 4, \ldots,  \frac{t-1+ 2(-1)^{\lfloor t/2 \rfloor}}{2} ;      \right.  \\
 & & \left. \hspace{33mm} \frac{t-2-(-1)^{\lfloor t/2 \rfloor} }{2} ;  \frac{t-3+2(-1)^{\lfloor t/2 \rfloor}}{2} ,\ldots,  3,t-5,1,t-3     \right) . \\
\end{eqnarray*} 
Prescott shows that $({\bf a_1}, {\bf a_2}, {\bf a_3})$ is a Roman triple~\cite{Prescott99}.
\end{proof}

Additionally, Prescott~\cite{Prescott99} shows that each of  the sequences used to prove Theorem~\ref{th:roman1} are as close to Roman as possible, in the sense that one non-zero element appears twice among the differences, another does not appear at all, and the rest appear exactly once each.

\begin{exa}\label{ex:z7vat}
Multiplication by~$2$ is an automorphism of~$\Z_7$ of order~$3$ with cycles~$\{ 1,2,4 \}$ and $\{3,5,6\}$.   A Vatican pseudoterrace with respect to this automorphism:
$$
\begin{array}{rcccccccccc}
{\bf a}: & 0 & 1 & 5 & 4 & 2 & 3 & 6 \\
\hline
T_1: && 1&4&6&5&1&3 \\
T_2: &&& 5&3&4&6&4\\
T_3: &&&&  4&1&5&2\\
T_4: &&&&&  2&2&1\\
T_5: &&&&&&  3&5\\
T_6: &&&&&&& 6\\
\end{array}
$$
Hence
$$ (0,1,5,4,2,3,6), (0,2,3,1,4,6,5), (0,4,6,2,1,5,3)  $$
is a Vatican triple.
\end{exa}

We now provide a number theoretic construction for pseudoterraces that are sometimes Roman-$k$ for $k > 1$, and sometimes even Vatican (although these tend to be $\ell$-fold with large~$\ell$).  Given a prime~$p$ and a primitive root~$\rho$ of~$p$. define the {\em primitive root construction} to be
$$0, \rho, \rho^2, \ldots, \rho^{p-1}.$$

\begin{thm}
The primitive root construction for a prime~$p$ with primitive root~$\rho$ is a Roman pseudoterrace with respect to multiplication by~$r = \rho / (\rho - 1)$.
\end{thm}

\begin{proof}  The elements 
$$\rho - 1, \rho^2 - \rho, \ldots, \rho^{p-1} - \rho^{p-2}$$
of $\Z_p$ are distinct.  The differences of the primitive root construction are exactly these elements, with the exception that $\rho-1$ is replaced by~$\rho$.   As $r (\rho-1) = \rho$, these two elements are in the same cycle with respect to~$r$ and the primitive root construction is a Roman pseudoterrace.
\end{proof}

When $\rho = (p+1)/2$ is a primitive root of $p$ we find that~$r=-1$ and so~$\ell = 2$. In this case the primitive root construction is the ``halving terrace" described in~\cite[Section~5]{OP03}, derived from~\cite[Theorem~2.1]{AP03}.

We are especially interested in determining when the primitive root construction gives a Roman-$k$ pseudoterrace for $k > 1$.  Trivially, if the pseudoterrace is $\ell$-fold for $\ell = p-1$ then it is  a Vatican pseudoterrace (indeed, any arrangement of the elements of~$\Z_p$ is a $(p-1)$-fold pseudoterrace for any automorphism of order~$p-1$).

In Table~\ref{tab:smallp} we compile the characteristics of pseudoterraces for all $\ell \mid p-1$ that give the most neighbor-balance for primes up to~$p=61$.  In Table~\ref{tab:largep} we give examples of Roman~$k$ pseudoterraces with $k >1$ from the primitive root construction for primes~$p$ in the range~$61 < p \leq 257$.  In each case, Vatican pseudoterraces are bolded.   For~$p$ in the range~$258 < p <1000$, here is a list of primitive root constructions $(p;\ell,k,\rho,r)$ that give Roman-$k$ pseudoterraces with~$k > 1$ and $\ell \leq 40$:

$$
\begin{array}{l}
{\bf (281; 35, 280, 187, 211)}, (281; 40, 3, 3, 142) ,  (307; 34, 3, 45, 8 ), (331; 11, 2, 3, 167) ,  \\
(337; 21, 3, 46, 16), (337;  28, 2, 154, 164 ) , (401;16, 2, 3, 202) , (419; 22, 2, 6, 85), \\
(431; 5, 2, 189, 95 ),  (443; 13, 2, 136, 339) , (463; 14, 2, 3, 233),  (521; 40, 2, 41, 509) , \\
(541; 36, 2, 409, 302 ),  (601; 30, 2, 254, 583 ), (613;9, 2, 163, 474 ), (701; 35, 2, 39, 536 ) , \\
(751; 30, 2, 39, 337 ),  (757;28, 3, 206, 710 ), (829;36, 2, 321, 444) , (991; 22, 4, 22, 237 ), \\
(991; 33, 2, 89, 733). \\ 
\end{array}
$$

Limiting to single-digit values of~$\ell$, in the range $1000 < p < 10000$, there is just one $\ell$-fold Roman-$k$ pseudoterrace with $k>1$ and $2 \leq \ell < 10$ (same format as previous list): $(2017; 9, 2, 1032, 1525)$.

\begin{table}[tp]
\begin{center}
\caption{Achieving the highest value of~$k$ in an $\ell$-fold Roman-$k$ pseudoterrace for~$\Z_p$ for non-trivial divisors~$\ell$ of~$p-1$ with the primitive root method.  For each prime~$p$ we give quadruples $(\ell,k,\rho,r)$.  Vatican pseudoterraces are given in bold. Values of~$\ell$ for which there is no successful primitive root construction are given as singletons~$(\ell)$.}\label{tab:smallp}
\begin{tabular}{rp{5.5in}}
\hline 
$p$ & $(\ell,k,\rho,r)$ \\
\hline 
5 & $\pmb{(2,4,3,4)}$ \\
7 & $(2), (3)$ \\
11 & $( 2, 1, 6, 10 ), \pmb{ (5, 10, 8, 9 )}$ \\
13 & $ ( 2, 1, 7, 12 ), ( 3, 1, 6, 9 ), ( 4, 1, 11, 5 ), ( 6 )$ \\
17 & $ ( 2 ), ( 4, 2, 7, 4 ), \pmb{( 8, 16, 12, 15 )}$ \\
19 & $ ( 2, 1, 10, 18 ), ( 3, 2, 3, 11 ), ( 6 ),  \pmb{(9, 18, 15, 16)} $ \\
23 & $ ( 2 ), \pmb{( 11, 22, 21, 16 )} $ \\
29 & $ ( 2, 1, 15, 28 ), ( 4, 1, 21, 17 ), \pmb{( 7, 28, 27, 20 )}, \pmb{( 14, 28, 19, 22 )} $ \\ 
31 & $ ( 2) , ( 3 ), ( 5, 1, 22, 4 ), (6), \pmb{ (10, 30, 21, 15 )}, \pmb{( 15, 30, 24, 28 )} $ \\
37 & $ ( 2, 1, 19, 36 ), ( 3 ), ( 4, 1, 22, 31 ), ( 6 ), ( 9, 1, 15, 9 ), \pmb{( 12, 36, 17, 8 )}, \pmb{(  18, 36, 24, 30 )} $ \\
41 & $ ( 2 ), ( 4 ), ( 5, 1, 12, 16 ), ( 8, 1, 11, 38 ), \pmb{( 10, 40, 29, 23 )}, \pmb{( 20, 40, 35, 36 )} $ \\ 
43 & $ ( 2 ), ( 3 ), ( 6 ), ( 7, 1, 20, 35 ), \pmb{ (14, 42, 18, 39 )}, \pmb{( 21, 42, 34, 31 )} $ \\
47 & $ ( 2 ), \pmb{( 23, 46, 45, 32 )} $ \\
53 & $ ( 2, 1, 27, 52 ), ( 4, 1, 12, 30 ), \pmb{( 13, 52, 51, 36 )}, \pmb{ (26, 52, 35, 40) }$ \\
59 & $ ( 2, 1, 30, 58 ), \pmb{( 29, 58, 56, 45 )} $ \\
61 & $ ( 2, 1, 31, 60 ),  ( 3 ), ( 4, 1, 6, 50 ), ( 5 ), ( 6 ), \pmb{( 10, 60, 30, 41 )}, ( 12, 1, 59, 21), ( 15, 2, 51, 12 ), $ \\
    & $   ( 20, 1, 26, 23 ), \pmb{( 30, 60, 54, 39 )} $ \\
\hline
\end{tabular}
\end{center}
\end{table}

\begin{table}[tp]
\begin{center}
\caption{Achieving  values of~$k>1$ in an $\ell$-fold Roman-$k$ pseudoterrace for~$\Z_p$ for non-trivial divisors~$\ell$ of~$p-1$ with the primitive root method.  For each prime~$p$ we give quadruples $(\ell,k,\rho,r)$.  Vatican pseudoterraces are given in bold.}\label{tab:largep}
\begin{tabular}{rp{5.5in}}
\hline 
$p$ & $(\ell,k,\rho,r)$ \\
\hline 
67 & $\pmb{ (22, 66, 50, 27)}, \pmb{ (33, 66, 46, 4 )} $\\
71 & $ \pmb{( 14, 70, 55, 26 )}, \pmb{ ( 35, 70, 67, 15 )} $ \\
73 & $ ( 9, 2, 34, 32 ) , \pmb{( 24, 72, 33, 17 )}, \pmb{(36, 72, 59, 35)} $ \\
79 & $ \pmb{( 13, 78, 75, 64 )}, \pmb{( 26, 78, 74, 14 )}, \pmb{( 39, 78, 70, 72 )} $ \\
83 & $ \pmb{( 41, 82, 80, 63 )} $ \\
89 & $ \pmb{( 22, 88, 30, 44 )}, \pmb{( 44, 88, 76, 20 )} $ \\
97 & $ ( 16, 2, 15, 8 ), \pmb{( 24, 96, 59, 93 )}, \pmb{( 32, 96, 87, 45 )}, \pmb{( 48, 96, 56, 31 )} $ \\ 
101 & $ \pmb{(20, 100, 48, 44 )},  \pmb{( 25, 100, 99, 68 )}, \pmb{ (50, 100, 94, 64 )} $ \\
103 & $ \pmb{( 34, 102, 84, 37 )}, \pmb{( 51, 102, 96, 91 )} $ \\
107 & $ \pmb{( 53, 106, 104, 81 )} $ \\
109 & $ \pmb{( 18, 108, 14, 43 )},\pmb{( 27, 108, 70, 80 )}, \pmb{ (36, 108, 103, 32 )}, \pmb{( 54, 108, 99, 100 )} $ \\
113 & $ \pmb{ ( 16, 112, 92, 78 )}, \pmb{( 28, 112, 76, 111 )}, \pmb{( 56, 112, 80, 104 )} $ \\
127 & $ ( 9, 2, 6, 52 ), (18, 3, 12, 105 ), \pmb{( 42, 126, 114, 10 )}, \pmb{( 63, 126, 112, 120 )} $ \\
131 & $ \pmb{ (26, 130, 95, 47 )}, \pmb{( 65, 130, 127, 27 )} $ \\
137 & $ \pmb{( 34, 136, 125, 22 )}, \pmb{( 68, 136, 114, 98 )} $ \\
139 & $ \pmb{( 46, 138, 119, 87 )}, \pmb{( 69, 138, 134, 24 )} $ \\
149 & $ \pmb{( 37, 148, 137, 127 )}, \pmb{( 74, 148, 147, 100 )} $ \\
151 & $ \pmb{( 25, 150, 134, 110 )}, \pmb{( 50, 150, 146, 26 )}, \pmb{( 75, 150, 141, 97 )} $ \\
157 & $ \pmb{( 26, 156, 21, 56 )}, \pmb{( 39, 156, 104, 126 )}, \pmb{( 52, 156, 123, 149 )}, \pmb{( 78, 156, 96, 120)} $ \\
163 & $ ( 27, 3, 19, 155 ) ,  \pmb{( 54, 162, 112, 48 )}, \pmb{( 81, 162, 148, 113 )} $ \\
167 & $ \pmb{( 83, 166, 165, 112 )} $ \\
173 & $ \pmb{ (43, 172, 166, 109 )}, \pmb{( 86, 172, 171, 116)} $ \\
179 & $ \pmb{(  89, 178, 176, 135 )} $ \\
181 &  $  \pmb{(36, 180, 171, 149 )}, \pmb{(  60, 180, 76, 71 )}, \pmb{( 90, 180, 163, 20 )} $ \\
191 & $ ( 19, 4, 58, 125 ), \pmb{( 38, 190, 148, 14 )}, \pmb{( 95, 190, 189, 128 )} $ \\ 
193 & $ ( 16, 2, 53, 27 ), \pmb{( 48, 192, 174, 165 )}, \pmb{( 64, 192, 188, 33 )}, \pmb{( 96, 192, 155, 95 )} $ \\
197 & $ \pmb{( 49, 196, 195, 132 )}, \pmb{( 98, 196, 185, 107 )} $ \\
199 & $ ( 33, 2, 38, 157 ), \pmb{( 66, 198, 176, 59 )}, \pmb{( 99, 198, 195, 160 )} $ \\
211 & $ ( 5, 2, 3, 107 ), (30, 2, 48, 10 ), \pmb{(42, 210, 155, 38)}, \pmb{(70, 210, 118, 102)}, \pmb{( 105, 210, 187, 136 )} $ \\
223 & $ \pmb{( 74, 222, 149, 111)}, \pmb{( 111, 222, 205, 177 )} $ \\
227 & $ \pmb{( 113, 226, 224, 171)} $ \\
229 & $  \pmb{(57, 228, 201, 151 )}, \pmb{( 76, 228, 205, 175 )}, \pmb{( 114, 228, 223, 99 )} $ \\
233 & $ \pmb{( 58, 232, 166, 210 )}, \pmb{( 116, 232, 213, 123 )} $ \\
239 & $ ( 17, 2, 42, 36 ), \pmb{( 34, 238, 156, 203 )}, \pmb{( 119, 238, 237, 160 )} $ \\
241 & $ ( 60, 2, 66, 90 ), \pmb{( 80, 240, 227, 17 )}, \pmb{( 120, 240, 204, 20 )} $ \\
251 & $ \pmb{( 50, 250, 29, 10 )}, \pmb{( 125, 250, 248, 189 )} $ \\
257 & $ ( 16, 2, 86, 128 ), \pmb{(  64, 256, 217, 95 )}, \pmb{( 128, 256, 252, 215 )} $ \\
\hline
\end{tabular}
\end{center}
\end{table}

The bolded entries in the table and lists above give rise to many new families of Vatican designs.  Theorem~\ref{th:psvat} collects those that are better in the, sense that $(p-1) / \ell$ is larger (in particular, it collects the instances with $(p-1) / \ell \geq 5$).

\begin{thm}\label{th:psvat}
There is an $\ell t \times t$ Vatican design for 
\begin{eqnarray*}
(t, \ell) & \in & \{ (61,10), (71,14), (79,13), (101,20), (109,18), (113,16),  (131,26),  (151,25), \\
                 &&       (157,26), (181,36), (191,38), (211,42), (239,34), (251,50), (281,35)  \}.
\end{eqnarray*}
\end{thm}

We now turn to pseudoterraces in small groups.  Let 
$$D_{2m} = \langle u,v : u^m = e = v^2, vu = u^{-1}v \rangle$$
be the dihedral group of order~$2m$, and let
 $$Q_{8} = \langle u,v : u^4 = e, v^2 = u^4, vu = u^{-1}v \rangle$$ 
be the quaternion group of order~8.

Table~\ref{tab:smgp} gives 2- and 3-fold Vatican pseudoterraces for all groups of order up to~11 in which they exist.

\begin{table}[tp]
\caption{Some $\ell$-fold Vatican pseudoterraces for small groups; $\ell \in \{2,3\}$.}\label{tab:smgp}

$$\begin{array}{rrll}
\hline
{\rm Group} & \ell & {\rm Aut.} & {\rm Pseudoterrace} \\
\hline
\Z_3 & 2 & 1 \mapsto 2 & (0,1,2) \\
\Z_4 & 2 & 1 \mapsto 3 & (0,1,3,2) \\
\Z_2^2 & 3 & 01 \mapsto 10  & (00,01,10,11) \\
                           &    & 10 \mapsto 11 &                       \\
\Z_5 & 2 & 1 \mapsto 4 & (0,1,3,2,4)  \\
\Z_6 & 2 & 1 \mapsto 5 & (0,1,4,2,3,5) \\
D_6 & 2 & u \mapsto u^2 & (e,  u^2v, u^2, u, v, uv)    \\
             &&     v \mapsto v &  \\
        & 3 &      u  \mapsto u &  (e, u, v, u^2,  u^2v, uv)     \\
             &&       v \mapsto u^2v &  \\
\Z_7 & 2 &  1 \mapsto  6 & (0,  1,3,6,4,5,2)  \\         
        & 3  & 1 \mapsto 2 & (0,1,3,2,5,6,4)\\    
\Z_8 & 2 & 1 \mapsto 3 & (0,1,3,5,2,6,7,4) \\
\Z_4 \times \Z_2 & 2 & 10 \mapsto 30 & (00, 01, 10, 21,   31, 11, 30, 20)    \\
     && 01 \mapsto 01 & \\
\Z_2^3 & 2 & 100 \mapsto 101 &  (000, 010, 100, 011, 111, 110, 101, 001)   \\
  && 010 \mapsto 010 & \\
  && 001 \mapsto 001 & \\
D_8 & 2 & u \mapsto u^3 &  (e, v, u^3v, u^2v, u, u^3, uv, u^2) \\
  && v \mapsto v & \\
Q_8 & 2 & u \mapsto u & (e, u, v, u^2v, u^3v, u^3, uv, u^2)   \\
  && v \mapsto u^2v &   \\
  & 3 & u \mapsto v & (e, u, u^2, v, uv, u^3v, u^2v, u^3)     \\
  && v \mapsto & u^3v  \\
\Z_9 & 2 & 1 \mapsto 8 & (0,1,4,2,7,5,6,3,8) \\
\Z_{10} & 2 & 1 \mapsto 9 & (0, 1,2,8,6,3,5,9,4,7) \\
D_{10} & 2 & u \mapsto u^4 & (e, v, u, u^2, u^3, u^4v, u^2v, u^4, uv, u^3v)   \\
 && v \mapsto v & \\
\Z_{11} & 2 & 1 \mapsto 10 & (0,1,3,6,10,7,5,4,9,2,8) \\

\hline
\end{array}
$$

\end{table}

Table~\ref{tab:smgp2} extends Table~\ref{tab:smgp} up to order~15, except that for brevity entries are limited to at most one 2- and 3-fold Vatican pseudoterrace at each order and orders~$n$ for which~$n+1$ is prime, and so a Vatican singleton exists, are omitted.

\begin{table}[tp]
\caption{More $\ell$-fold Vatican pseudoterraces for small groups; $\ell \in \{2,3\}$.}\label{tab:smgp2}

$$\begin{array}{rrll}
\hline
{\rm Group} & \ell & {\rm Aut.} & {\rm Pseudoterrace} \\
\hline
\Z_{13} & 2 & 1 \mapsto 12 & (0,1,3,4,9,6,10,2,8,11,7,5,12) \\
 & 3 &1 \mapsto 3 & (0,1,2,,3,10,4,8,7,12,5,9,11,6) \\
 \Z_{14} & 2 & 1 \mapsto 13 & (0,1,5,2,8,6,12,3,10,9,11,7,4,13) \\
  & 3 & 1 \mapsto 9 & (0,1,2,3,5,13,9,12,4,11,10,6,8,7) \\
\Z_{15} & 2 & 1 \mapsto 14 & (0,1,4,10,8,2,7,12,9,13,11,3,14,6,5) \\
\hline
\end{array}
$$

\end{table}

As~11 is prime, there is a 1-fold Vatican pseudoterrace (that is, a directed $T_{9}$-terrace) for~$\Z_{10}$, and so the lack of a 3-fold Vatican pseudoterrace for a group of order~10 is not detrimental to the construction of Vatican designs.  However, for completeness, here is a 5-fold Vatican pseudoterrace for $D_{10}$ with automorphism $u \mapsto u, \ v \mapsto u^4v$:
$$( e, v, u^4v, u^2v, u^2, uv, u, u^4, u^3, u^3v ).$$
This implies the existence of a $10\ell  \times 10$ Vatican design built from Cayley tables of~$D_{10}$ when~$\ell \not\in \{1,3 \}$.  

We run into the same issue at~$t=11$, except that here we do not have that~$t+1$ is prime to construct the desired Vatican designs.  The primitive root construction gives a 5-fold Vatican pseudoterrace for~$\Z_{11}$ using the primitive root~$\rho = 8$, which is sufficient to show that there is an $11 \ell  \times 11 $ Vatican design when~$\ell \not\in \{1,3 \}$.  The $\ell=3$ case is covered below.

The gaps in the tables are genuine.   There is no 2-fold Vatican pseudoterrace for~$\Z_3 \times \Z_3$ despite this group having an automorphism of order~2 (the same is true for~$\Z_{17}$).  Similarly, there are no 3-fold Vatican pseudoterraces for~$\Z_9$ or~$\Z_3 \times \Z_3$.  There is no 3-fold Vatican pseudoterrace (or even an $\ell$-fold one for any odd~$\ell$) for~$\Z_3$ or~$\Z_{15}$ as these groups have no automorphisms of odd order.

Finally, to complete the proof of Theorem~\ref{th:vat}, Table~\ref{tab:direct} gives some direct constructions of Vatican triples.

\begin{table}[tp]
\caption{Some Vatican triples}\label{tab:direct}

$$\begin{array}{rl}
\hline
{\rm Group} & {\rm Triple} \\
\hline
\Z_{5}  & ( 0, 1, 2, 4, 3 ), ( 0, 2, 1, 4, 3 ), ( 0, 2, 3, 1, 4 ) \\
\Z_9 & ( 0, 1, 2, 3, 6, 8, 5, 4, 7 ), ( 0, 4, 6, 3, 2, 7, 5, 1, 8 ), ( 0, 7, 1, 5, 2, 6, 8, 4, 3 ) \\
\Z_{11} &  ( 0, 1, 2, 3, 5, 7, 4, 10, 9, 8, 6 ), ( 0, 2, 9, 5, 10, 6, 3, 8, 1, 4, 7 ), \\ 
 &  ( 0, 5, 8, 1, 7, 6, 4, 2, 10, 3, 9 ) \\
\hline
\end{array}
$$

\end{table}

\end{document}